\documentclass{amsart}
\usepackage{amssymb}
\usepackage{hyperref}
\hypersetup{hidelinks}
\hypersetup{
  pdftitle={Separation Conditions for Iterated Function Systems with a Common Attractor},
  pdfauthor={Cai-Yun Ma},
  pdfsubject={Research article},
  pdfkeywords={Self-similar sets, iterated function systems, open set condition, strong separation condition, chain components}
}

\numberwithin{equation}{section}
\newtheorem{theorem}{Theorem}[section]
\newtheorem{corollary}[theorem]{Corollary}
\newtheorem{lemma}[theorem]{Lemma}
\theoremstyle{definition}
\newtheorem{definition}[theorem]{Definition}
\newtheorem{question}[theorem]{Question}

\theoremstyle{remark}

\newcommand{\R}{\mathbb{R}}

\newcommand{\cH}{\mathcal{H}}
\newcommand{\dimH}{\dim_{\mathrm H}}
\DeclareMathOperator{\diam}{diam}
\DeclareMathOperator{\dist}{dist}

 \title[{\protect\fontsize{5}{6}\protect\selectfont Separation Conditions for Iterated Function Systems with a Common Attractor}]
{Separation Conditions for Iterated Function Systems with a Common Attractor}
\author[C.-Y.~Ma]{Cai-Yun Ma}
\address[C.-Y Ma]{HUN-REN-BME Stochastics Research Group, M\H{u}egyetem rkp. 3., H-1111 Budapest, Hungary}
\email{macaiyun@edu.bme.hu}
\date{}

\begin{document}
\begin{abstract}
	Let $K$ be the non-singleton attractor of a finite iterated function
	system
	$\Phi=\{\phi_i(x)=a_i O_i x+t_i\}_{i=1}^m$
	on $\mathbb{R}^d$, where   $0<a_i<1$,
	$t_i\in\mathbb{R}^d$  and the $O_i$ are orthogonal transformations
	on $\mathbb{R}^d$.
	Suppose that $\Phi$ satisfies the open set condition, but not
	the strong separation condition.
	We show that $K$ cannot be generated by any finite iterated
	function system $\Psi$ of similitudes satisfying the strong separation
	condition.
	This removes the homogeneity assumption from a result of
	Feng, Ruan and Xiong and gives a negative answer to a folklore
	question about the separation conditions on the generating
	iterated function systems of self-similar sets.
\end{abstract}

\maketitle
\pagestyle{plain}
\thispagestyle{plain}

\section{Introduction}\label{sec:introduction}

We study separation conditions for iterated function systems of
similitudes that generate the same self-similar set.

Fix an integer $d\geq1$. An iterated function system (IFS) on $\R^d$
is a finite family $\Phi=\{\phi_i\}_{i=1}^m$ of contracting maps,
where $m\geq2$.    By Hutchinson's theorem~\cite{Hutchinson1981},
there is a unique non-empty compact set $K\subset\R^d$ such that
\begin{equation*}
K=\bigcup_{i=1}^m \phi_i(K).
\end{equation*}
  The set $ K $ is called the
\emph{attractor} of $\Phi$. We restrict attention to IFSs of
similitudes, whose maps have the form
\begin{equation*}  	\phi_i(x)=a_iO_ix+t_i,\qquad i=1,\ldots,m, \end{equation*}
where $0<a_i<1$, $t_i\in\mathbb{R}^d$ and $O_i$ is an orthogonal transformation on $\mathbb{R}^d$.
In this case $K$ is a \emph{self-similar set}, and $\Phi$ is called
a \emph{generating IFS} of $K$. Throughout the paper, generating
IFSs are finite and consist  of contracting similitudes.
For background on self-similar sets and measures, we refer to
\cite{BaranySimonSolomyak2023,BishopPeres2017,Falconer2003}.

We recall two     fundamental separation conditions. The IFS $\Phi$
satisfies the \emph{open set condition} (OSC) if there is a non-empty
open set $U\subset\R^d$ such that   $ \phi_i(U) $,$  i = 1,\ldots,m, $
are disjoint subsets of $ U $.
The IFS satisfies
the \emph{strong separation condition} (SSC) if $ \phi_i(K) $, $ i= 1,\ldots,m $, are pairwise disjoint.

The SSC implies the OSC. Indeed, let
$\delta=\min_{i\neq i'}\dist(\phi_i(K),\phi_{i'}(K))>0$ and let
$U_\varepsilon=\{x\in\R^d:\dist(x,K)<\varepsilon\}$. If
$0<\varepsilon<\delta/2$, then $\phi_i(U_\varepsilon)\subset U_\varepsilon$
and the sets $\phi_i(U_\varepsilon)$ are pairwise disjoint.

The converse fails, even for totally disconnected attractors.    For example, the IFS
$\{x/5,(x+3)/5,(x+4)/5\}$ on $\R$ satisfies the OSC with feasible
open set $(0,1)$, but its last two pieces meet at $4/5$;
see~\cite{FengRuanXiong2023}.

Let $s_\Phi>0$ be the \emph{similarity dimension} of $\Phi$, defined by
\begin{equation}\label{eq:similarity-dimension}
	\sum_{i=1}^m a_i^{s_\Phi}=1.
\end{equation}
Under the OSC, Hutchinson's theorem~\cite[Section~5.3]{Hutchinson1981}
gives
\begin{equation}\label{eq:osc-dimension}
	\dimH K=s_\Phi
	\qquad\text{and}\qquad
	0<\cH^{s_\Phi}(K)<\infty,
\end{equation}
where $\dimH$ and $\cH^t$ denote Hausdorff dimension and
$t$-dimensional Hausdorff measure, respectively.
In particular, the preceding example has dimension $\log3/\log5<1$
and is totally disconnected. Further characterizations of the OSC
were obtained by Bandt and Graf~\cite{BandtGraf1992} and
Schief~\cite{Schief1994}. Schief showed that the OSC is equivalent
to $\cH^{s_\Phi}(K)>0$, though this criterion is often difficult
to check and depends on the similarity dimension of the given IFS.

A self-similar set may admit more than one generating IFS.
The algebraic restrictions on such systems have been studied
in~\cite{FengWang2009,ElekesKeletiMathe2010,DengLau2013,
	DengLau2017,Algom2020}. Our concern here is whether different
generating systems can have different separation properties.
More precisely, we consider the following folklore question
recorded in~\cite[Question~1.1]{FengRuanXiong2023}.

\medskip
\noindent\textbf{Question.}
\emph{Can a self-similar set admit both an SSC generating IFS
	and another generating IFS satisfying the OSC but not the SSC?}
\medskip

As reported in~\cite{FengRuanXiong2023}, Ka-Sing Lau and Jun Jason Luo
brought this question to the attention of Feng, Ruan and Xiong.
Lau and Luo also informed them that Mariusz Urba\'nski had asked the
same question in a private communication.

	Feng, Ruan and Xiong~\cite[Theorem~1.2]{FengRuanXiong2023}
	answered this question negatively when the OSC generating IFS is
	\emph{homogeneous}, that is, when all its contraction ratios are equal.
	The orthogonal parts may vary, and the SSC generating IFS need not
	be homogeneous.

	Our main result removes this homogeneity assumption and gives a
	negative answer for arbitrary finite IFSs of contracting similitudes.

\begin{theorem}\label{thm:main}
	Let $K$ be the non-singleton attractor of a finite IFS
	$\Phi=\{\phi_i\}_{i=1}^m$ of contracting similitudes on
	$\mathbb{R}^d$, $d\geq 1$.
	Suppose that $\Phi$ satisfies the OSC, but not the SSC.
	Then $K$ cannot be generated by any finite IFS $\Psi$ of contracting
	similitudes on $\mathbb{R}^d$ satisfying the SSC.
\end{theorem}
In particular, if a self-similar set admits an SSC generating IFS,
then every OSC generating IFS of that set must also satisfy the SSC.
Together with the standard dimension and measure criteria for the
OSC, this yields the following characterization.
\begin{corollary}\label{cor:equivalent-criteria}
	Let $\Phi=\{\phi_i \}_{i=1}^m$ be a finite IFS of contracting
	similitudes on $\R^d$ with non-singleton attractor $K$ and similarity dimension $s_\Phi$.   Write $s= \dimH K $   and suppose  that  	$K$ admits a generating IFS satisfying the SSC.
	Then $\Phi$ satisfies the SSC if and only if $s_\Phi=s$.
	If $\Phi$ does not satisfy the SSC, then $s_\Phi>s$ and there exist
	$1\leq i<i'\leq m$ such that
	\begin{equation*}
	\cH^s\bigl(\phi_i(K)\cap \phi_{i'}(K)\bigr)>0.
\end{equation*}
\end{corollary}

A related question was raised by Elekes, Keleti and
M\'ath\'e~\cite[Question~9.3]{ElekesKeletiMathe2010}.

\medskip
\noindent\textbf{Question.}
\emph{If $K$ is generated by an IFS satisfying the SSC and $h$ is a
	similitude with $h(K)\subseteq K$, must $h(K)$ be relatively open
	in $K$?
	Equivalently, must $h(K)$ be a finite union of cylinder sets $ 	\psi_{j_1}\circ\cdots\circ \psi_{j_\ell}(K) $
	of the given SSC generating IFS $\{\psi_j\}_{j=1}^n$?}
\medskip

As explained in~\cite[Section~4]{FengRuanXiong2023}, an affirmative
answer to this question would also rule out the existence of two
generating systems as in the separation problem. Xiao~\cite{Xiao2024}
proved the relative openness property when the generating maps have
a common contraction ratio and a common orthogonal part, but
constructed a counterexample when the orthogonal parts vary. Thus the relative openness question has a negative answer in general.

Nevertheless, Theorem~\ref{thm:main} gives a precise characterization
of those contracting self-embeddings whose images are relatively open
in $K$. Equivalently, it determines exactly which self-embeddings can
occur as maps of an OSC generating IFS. This yields the following
corollary.
\begin{corollary}\label{cor:embedding-completion}
	Let $K\subset\R^d$ be a non-singleton self-similar set admitting a
generating IFS satisfying the SSC.
	Let $h:\mathbb{R}^d\to\mathbb{R}^d$ be a contracting similitude
	such that $h(K)\subseteq K$.
	Then the following are equivalent:
	\begin{enumerate}
		\item [(i)] $h(K)$ is relatively open in $K$;
		\item [(ii)] there exists a finite generating IFS of contracting
		similitudes of $K$ satisfying the OSC that contains $h$;
		\item [(iii)] there exists a finite generating IFS of contracting
		similitudes of $K$ satisfying the SSC that contains $h$.
	\end{enumerate}
	Moreover, if one of these conditions holds, then for every fixed SSC
	generating IFS $\Psi=\{\psi_j\}_{j=1}^n$ of $K$, there exist an
	integer $\ell\geq1$ and a set
	$\mathcal{J}\subseteq\{1,\ldots,n\}^{\ell}$ such that
	\begin{equation*}
	\{h\}\cup
	\bigl\{
	\psi_{j_1}\circ\cdots\circ \psi_{j_\ell}:
	(j_1,\ldots,j_\ell)\in\mathcal{J}
	\bigr\}
\end{equation*}
	is a generating IFS of $K$ satisfying the SSC.
\end{corollary}

Consequently, although a contracting self-embedding of an SSC
self-similar set need not have relatively open image, precisely those
self-embeddings with relatively open image can occur as maps in an
OSC generating IFS. In particular, a self-embedding whose image is
not relatively open in $K$, such as the one constructed
in~\cite{Xiao2024}, cannot belong to any OSC generating IFS of $K$.

  	We briefly compare our argument with that of Feng, Ruan and
  	Xiong~\cite{FengRuanXiong2023}.   Their proof uses characteristic
  	vectors to record the intersection structure of generating systems.
  	Assuming an SSC generating system exists, a finiteness consequence
  	of a theorem of Elekes, Keleti and M\'ath\'e~\cite{ElekesKeletiMathe2010}
  	allows them to choose a maximal characteristic vector among homogeneous OSC generating systems whose contraction ratios lie in a	fixed interval bounded away from zero.  Homogeneity allows auxiliary
  	systems to return to this interval by composition with suitable
  	iterates of the non-SSC system. The strict increase of characteristic
  	vectors under composition then yields a contradiction. For unequal
  	contraction ratios, fixed-level iteration no longer provides a common
  	rescaling, while refining individual cylinders need not preserve the
  	required monotonicity. Thus the difficulty lies in the scale
  	normalization, rather than in the composition inequality itself,
  	which does not require homogeneity.

  	Our proof instead uses the number $N(r)=N_K(r)$ of $r$-chain
  	components of $K$; see Section~\ref{sec:components}. This function
  	depends only on the attractor and the scale, rather than on a
  	particular generating system. Suppose, to the contrary, that
  	$K$ is generated by an OSC system $\Phi=\{\phi_i\}_{i=1}^m$
  	which fails the SSC and an SSC system $\Psi=\{\psi_j\}_{j=1}^n$.
  	Let $a_i$ and $b_j$ be their contraction ratios, respectively. Both systems
  	have similarity dimension $s=\dimH K$, so $\sum_i a_i^s=\sum_j b_j^s=1$. The SSC gives exact additivity of
  	the component count at small scales, whereas an intersection 	between  two distinct $ \Phi $-pieces  causes a strict loss when their  	components are counted in the union. Integrating the resulting	counting inequality between two positive scales leads us to	consider the auxiliary function $ F $  defined   in \eqref{def-ft}. We show that $F$ is strictly increasing	near zero. On the other hand, exact additivity for the SSC system	gives
  	\begin{equation*}  F(t)=\sum_{j=1}^n b_j^sF(t/b_j) \end{equation*}  for all sufficiently small
  	$t>0$. This is impossible: the
  	right-hand side is a convex combination of values strictly larger	than $F(t)$.  The argument uses only integrals over	positive scale intervals and does not require a common contraction ratio.

  The paper is organized as follows. Section~\ref{sec:components} introduces chain components and establishes
  the counting inequalities used in the proof. In Section~\ref{sec:proof}, we combine these inequalities with  an integral argument to prove Theorem~\ref{thm:main}.
Section~\ref{sec:consequences} is devoted to the proofs of
Corollaries~\ref{cor:equivalent-criteria} and \ref{cor:embedding-completion}.
Finally, Section~\ref{sec:questions} poses a question about graph-directed systems.

\section{Chain components and counting inequalities}\label{sec:components}

In this section, we introduce the notion of chain components and the
associated component-counting function used by Rao, Ruan and Yang
\cite{RaoRuanYang2008}. We then prove two elementary lemmas to be used
in the proof of Theorem~\ref{thm:main}.

\begin{definition}
	Let $A\subset\R^d$ be a non-empty compact set and let $r>0$.
	We say that $A$ is \emph{$r$-connected} if for any $x,y\in A$,
	there is an $r$-chain connecting $x$ and $y$. That is, there
	are an integer $\ell\geq0$ and points
	$x=x_0,x_1,\ldots,x_\ell=y$ in $A$ such that
	$|x_{k+1}-x_k|\leq r$ holds for $0\leq k<\ell$.

	We call a non-empty subset $F\subseteq A$ an \emph{$r$-component}
	of $A$ if $F$ is $r$-connected, but for any
	$F\subsetneq F'\subseteq A$, $F'$ is not $r$-connected.
	Let us denote by $N_A(r)$ the number of $r$-components of $A$,
	which is finite by the compactness of $A$.
\end{definition}

By \cite[Lemma~1]{RaoRuanYang2008}, for every non-empty compact
set $A\subseteq\R^d$, the function
$N_A:(0,\infty)\to\mathbb N$ is non-increasing,
right-continuous, and locally constant at every point of continuity.
In particular, $N_A$ is Borel measurable and locally bounded on
$(0,\infty)$. Moreover,
\begin{equation}\label{eq:large-scale}
	N_A(r)=1\qquad\text{for }r\geq\diam A,
\end{equation}
where $\diam A:=\sup\{|x-y|:x,y\in A\}$ denotes the diameter of $A$.
For any similitude $h:\R^d\to\R^d$ with similarity ratio $c>0$,
we have
\begin{equation}\label{eq:scaling}
	N_{h(A)}(r)=N_A(r/c)\qquad\text{for }r>0.
\end{equation}
To see this, one simply applies $h$ and $h^{-1}$ to the chains in
the definition.

We first record the subadditivity of the component count under finite
unions and its additivity under a separation condition.
\begin{lemma}\label{lem:union}
	Let $r>0$ and let $q\geq1$ be an integer.
	Let $A_1,\ldots,A_q$ be non-empty compact subsets of $\R^d$
	and set $A=\bigcup_{k=1}^q A_k$.
	Then the following statements hold.
	\begin{itemize}
		\item [(i)]	We have $ 	N_A(r)\leq\sum_{k=1}^q N_{A_k}(r). $
	 	\item	[(ii)]		If $A_k\cap A_{k'}\neq\varnothing$ for some $k\neq k'$,
		then
	\begin{equation*}  			N_A(r)\leq\sum_{k=1}^q N_{A_k}(r)-1. \end{equation*}

		\item [(iii)]
		If $\dist(A_k,A_{k'})>r$ for all $k\neq k'$, then
	\begin{equation*}	N_A(r)=\sum_{k=1}^q N_{A_k}(r).   \end{equation*}
	\end{itemize}
\end{lemma}

\begin{proof}
	Let $\mathcal C_r(E)$ denote the family of $r$-components of $E\subset \R^d$.
	For $C\in\mathcal C_r(A)$, set
	\begin{equation*}
	M(C)=\sum_{k=1}^q
	\#\{F\in\mathcal C_r(A_k):F\subseteq C\}.
	\end{equation*}
	Each $r$-component of $A_k$ lies in a unique $r$-component of $A$,
	and each $r$-component of $A$ contains at least one $r$-component from some $A_k$.
	Thus $M(C)\geq1$ and
	\begin{equation}\label{eq-MC}
		\begin{aligned}
			\sum_{k=1}^q N_{A_k}(r)-N_A(r)
			&=\sum_{C\in\mathcal C_r(A)}M(C)-\#\mathcal C_r(A)\\
			&=\sum_{C\in\mathcal C_r(A)}(M(C)-1)\geq0,
		\end{aligned}
	\end{equation}
	which proves (i).

	If $x\in A_k\cap A_{k'}$ with $k\neq k'$, let $C(x)$ be the
	$r$-component of $A$ containing $x$.
	The components of $A_k$ and $A_{k'}$ containing $x$ both lie
	in $C(x)$ and are counted separately, so $M(C(x))\geq2$.
	By \eqref{eq-MC},
	\begin{equation*}
	\sum_{k=1}^q N_{A_k}(r)-N_A(r)
	\geq M(C(x))-1\geq1,
	\end{equation*}
	which proves (ii).

	Finally, if $\dist(A_k,A_{k'})>r$ for all $k\neq k'$, every
	$r$-chain in $A$ lies entirely in one $A_k$.
	Hence $M(C)=1$ for every $C\in\mathcal C_r(A)$,
	and (iii) follows   from \eqref{eq-MC}.
\end{proof}

We next compare two generating systems of the same set.
Let $K\subset\R^d$ be a non-singleton compact set which is the
attractor of two finite IFSs $\Phi=\{\phi_i\}_{i=1}^m$ and
$\Psi=\{\psi_j\}_{j=1}^n$ of contracting similitudes.
Let $a_i$ and $b_j$ denote the contraction ratios of $\phi_i$ and $\psi_j$,
respectively, and write $N(r)=N_K(r)$ for $r>0$.

	\begin{lemma}\label{lem:strict-counting}
		Suppose that $\Psi$ satisfies the SSC, whereas $\Phi$
		does not satisfy the SSC. Set
		\begin{equation*}
		\delta:=\min_{1\leq j\neq k\leq n}
		\dist\bigl(\psi_j(K),\psi_k(K)\bigr),
		\qquad a_*:=\min_{1\leq i\leq m}a_i,
		\qquad \delta_*:=a_*\delta.
		\end{equation*}
		Then
		\begin{equation}\label{eq:nbj}
			N(r)=\sum_{j=1}^n N(r/b_j)  \quad\text{for } 0<r<\delta,
		\end{equation}
	and
		\begin{equation}\label{eq:strict-counting}
			\sum_{i=1}^m N(r/a_i)\geq N(r)+n, \quad\text{for } 0<r<\delta_*.
		\end{equation}
	\end{lemma}

	\begin{proof}
		Since $\Psi$ satisfies the SSC,   \eqref{eq:scaling} and Lemma~\ref{lem:union}(iii) give \eqref{eq:nbj} for every $ 0<r<\delta $.

		On the other hand, since $\Phi$ does not satisfy the SSC, \eqref{eq:scaling} and Lemma~\ref{lem:union}(ii) give
		\begin{equation}\label{eq-nonssc}
			\sum_{i=1}^m N(r/a_i)\geq N(r)+1
			\qquad\text{for every }r>0.
		\end{equation}

		Fix $0<r<\delta_*$. Then $r/a_i<\delta$ for every $i$.
		Applying \eqref{eq:nbj} at each scale $r/a_i$
		and summing over $i$, we obtain
		\begin{align*}
			0
			&=\sum_{i=1}^m
			\left(\sum_{j=1}^n N\!\left(\frac{r}{b_j a_i}\right)
			-N(r/a_i)\right)\\
			&=\sum_{j=1}^n\sum_{i=1}^m
			N\!\left(\frac{r/b_j}{a_i}\right)
			-\sum_{i=1}^m N(r/a_i)\\
			&\geq\sum_{j=1}^n\bigl(N(r/b_j)+1\bigr)
			-\sum_{i=1}^m N(r/a_i) \quad \text{(by  \eqref{eq-nonssc})}\\
			&=N(r)+n-\sum_{i=1}^m N(r/a_i),
		\end{align*}
		which  yields \eqref{eq:strict-counting}.
	\end{proof}

\section{Proof of the main theorem}\label{sec:proof}

In this section, we prove Theorem~\ref{thm:main} by combining the counting inequalities from Section~\ref{sec:components} with
integration over intervals bounded away from zero.

\begin{proof}[Proof of Theorem~\ref{thm:main}]
	Suppose, to the contrary, that $K$ is also the attractor of a finite
	IFS $\Psi=\{\psi_j\}_{j=1}^n$ of contracting similitudes satisfying
	the SSC. Let $a_i$ and $b_j$ denote the contraction ratios
	of $\phi_i$ and $\psi_j$, respectively. Since $K$ is not a singleton,
	$m,n\geq2$.

	Since the SSC implies the OSC, both $\Phi$ and $\Psi$ satisfy
	the OSC. By \eqref{eq:osc-dimension}, $s:=s_\Phi=s_\Psi=\dimH K>0$, and
	\begin{equation}\label{eq:critical-weights}
		\sum_{i=1}^m a_i^s=\sum_{j=1}^n b_j^s=1.
	\end{equation}
	Let $a_*$, $\delta$ and $\delta_*$ be defined as in
	Lemma~\ref{lem:strict-counting}, and write $N(r)=N_K(r)$
	for $r>0$.

	Let $0<\varepsilon<R<\delta_*$.
	Multiplying \eqref{eq:strict-counting} by $r^{s-1}$ and integrating, we obtain
	\begin{equation}\label{eq:integrated-counting}
		\int_\varepsilon^R r^{s-1}
		\left(\sum_{i=1}^m N(r/a_i)-N(r)\right)\,dr
		\geq\frac{n}{s}(R^s-\varepsilon^s).
	\end{equation}
The integral on the left-hand side of
\eqref{eq:integrated-counting} is finite, since $N$ is locally
bounded on $(0,\infty)$.	Changing variables and using $\sum_i a_i^s=1$, we obtain
	\begin{equation}\label{eq-equna}
		\begin{aligned}
			&\int_\varepsilon^R r^{s-1}
			\left(\sum_{i=1}^m N(r/a_i)-N(r)\right)\, dr\\
			& =\sum_{i=1}^m a_i^s
			\int_{\varepsilon/a_i}^{R/a_i}r^{s-1}N(r)\, dr
			-\int_\varepsilon^R r^{s-1}N(r)\,dr\\
			& =\sum_{i=1}^m a_i^s \left(
			\int_{\varepsilon/a_i}^{R/a_i}r^{s-1}N(r)\, dr
			-\int_\varepsilon^R r^{s-1}N(r)\,dr\right)\\
			& =\sum_{i=1}^m a_i^s
			\left(\int_R^{R/a_i}r^{s-1}N(r)\,dr
			-\int_\varepsilon^{\varepsilon/a_i}
			r^{s-1}N(r)\, dr\right).
		\end{aligned}
	\end{equation}
	Define
	\begin{equation}\label{def-ft}
		F(t):=\sum_{i=1}^m a_i^s
		\int_t^{t/a_i}r^{s-1}N(r)\,dr,
		\qquad 0<t<\delta_*.
	\end{equation}
	Since $N$ is non-increasing, we have
	\begin{align*}
		0\leq F(t)
		&\leq N(t)\sum_{i=1}^m a_i^s
		\int_t^{t/a_i}r^{s-1}\,dr\\
		&=\frac{m-1}{s}\,t^sN(t)<\infty.
	\end{align*}
	Combining \eqref{eq-equna} and \eqref{def-ft} with
	\eqref{eq:integrated-counting}, we obtain
	\begin{equation}\label{eq:boundary-increasing}
		F(R)-F(\varepsilon)
		\geq\frac{n}{s}(R^s-\varepsilon^s)
		\qquad\text{for }0<\varepsilon<R<\delta_*.
	\end{equation}
	In particular, $F$ is strictly increasing on $(0,\delta_*)$.

	Set $b_*:=\min_{1\leq j\leq n}b_j$ and fix
	$0<t<b_*\delta_*$. Then $t<t/b_j<\delta_*$ for every $j$.
	For every $i$ and every $r\in[t,t/a_i]$, we have
	$r\leq t/a_*<\delta$, so \eqref{eq:nbj} applies throughout
	the interval of integration. Substituting \eqref{eq:nbj}
	into \eqref{def-ft} gives
	\begin{align*}
		F(t)
		&=\sum_{i=1}^m a_i^s
		\int_t^{t/a_i}r^{s-1}
		\sum_{j=1}^n N(r/b_j)\,dr\\
		&=\sum_{j=1}^n b_j^s\sum_{i=1}^m a_i^s
		\int_{t/b_j}^{t/(b_j a_i)}r^{s-1}N(r)\, dr\\
		&=\sum_{j=1}^n b_j^s F(t/b_j).
	\end{align*}
	On the other hand, since $t<t/b_j<\delta_*$, the strict
	monotonicity of $F$ implies that $F(t/b_j)>F(t)$ for every $j$.
	Together with \eqref{eq:critical-weights}, this yields
	\begin{equation*}
	F(t)=\sum_{j=1}^n b_j^sF(t/b_j)
	>\sum_{j=1}^n b_j^sF(t)=F(t),
	\end{equation*}
	which is a contradiction. Therefore, no such IFS $\Psi$ exists.
\end{proof}

 \section{Proofs of the corollaries}\label{sec:consequences}

In this section, we   prove Corollaries~\ref{cor:equivalent-criteria}
and~\ref{cor:embedding-completion}.

\begin{proof}[Proof of Corollary~\ref{cor:equivalent-criteria}]
	Since $K$ admits an SSC generating IFS, we have $s>0$ and
	$0<\cH^s(K)<\infty$ by~\eqref{eq:osc-dimension}. If $\Phi$ satisfies
	the SSC, then it satisfies the OSC, and~\eqref{eq:osc-dimension}
	gives $s_\Phi=s$.

	Suppose now that $\Phi$ does not satisfy the SSC. We first show that  $ s_\Phi>s $. By the scaling property of Hausdorff measure,
	\begin{equation*}
	\cH^s(K)
	\leq\sum_{i=1}^m\cH^s(\phi_i(K))
	=\cH^s(K)\sum_{i=1}^m a_i^s.
\end{equation*}
	Hence $\sum_{i=1}^m a_i^s\geq1$. We claim that the inequality is strict. Indeed, if  $\sum_{i=1}^m a_i^s =1$,  then $s_\Phi=s$, and therefore
	$\cH^{s_\Phi}(K)=\cH^s(K)>0$.   By
	Schief's theorem~\cite{Schief1994}, $\Phi$   satisfies the OSC,
	and then by Theorem~\ref{thm:main}, $\Phi$ satisfies the SSC, which is a contradiction. Thus $\sum_{i=1}^m a_i^s>1$,  and consequently $s_\Phi>s$.

	Finally, suppose that $ 	\cH^s\bigl(\phi_i(K)\cap \phi_{i'}(K)\bigr)=0 $ for all $ 1\leq i<i' \leq m. $
Then
	\begin{equation*}
	\cH^s(K)
	=\sum_{i=1}^m\cH^s(\phi_i(K))
	=\cH^s(K)\sum_{i=1}^m a_i^s,
\end{equation*}
Since $0<\cH^s(K)<\infty$, this implies
$  \sum_{i=1}^m a_i^s=1,$
contradicting the strict inequality obtained above. Hence
\begin{equation*}
\cH^s\bigl(\phi_i(K)\cap \phi_{i'}(K)\bigr)>0
\end{equation*}
for some distinct $i,i'$. This completes the proof.
\end{proof}

\begin{proof}[Proof of Corollary~\ref{cor:embedding-completion}]
We prove the equivalence of the three conditions and establish the
final assertion in the proof of   (i) $\Rightarrow$ (iii).

	(iii)  $ \Rightarrow $ (ii). This follows since the SSC implies the OSC.

	(ii)  $ \Rightarrow $ (iii).
	Suppose that $h$ belongs to a finite generating IFS of contracting
	similitudes of $K$ satisfying the OSC.
	Since $K$ also admits an SSC generating IFS,
	Theorem~\ref{thm:main} implies that this system satisfies the SSC.

	(iii)  $ \Rightarrow $ (i).
	Let $\{\psi_j\}_{j=1}^n$ be an SSC generating IFS of $K$ containing $h$. The complement of $h(K)$ in $K$ is a finite union of the
	other compact first-level pieces. It is therefore closed in $K$,
	so $h(K)$ is relatively open in $K$.

	(i)  $ \Rightarrow $ (iii).  Suppose that $h(K)$ is relatively open
	in $K$. We construct an SSC generating IFS containing $h$ and
	having the form stated in the final assertion.

	Fix any SSC
	generating IFS $\Psi=\{\psi_j\}_{j=1}^n$ of $K$.
	Let $c_h\in(0,1)$ be the contraction ratio of $h$, let
	$b_j\in(0,1)$ be that of $\psi_j$, and let
	$b_{\max}=\max_{1\leq j\leq n}b_j$.
	The sets $h(K)$ and $K\setminus h(K)$ are disjoint non-empty
	compact sets. Indeed, the complement is non-empty since
	\begin{equation*}
	\diam h(K)=c_h\diam K<\diam K.
\end{equation*}
	Therefore
	\begin{equation*}
	\eta=\dist\bigl(h(K),K\setminus h(K)\bigr)>0.
\end{equation*}
	Choose an integer $\ell\geq1$ such that
	$b_{\max}^{\ell}\diam K<\eta$.
	For $w=(j_1,\ldots,j_\ell)\in\{1,\ldots,n\}^{\ell}$, write
	$\psi_w=\psi_{j_1}\circ\cdots\circ \psi_{j_\ell}$.
	By the SSC, the sets $\psi_w(K)$ form a finite partition of $K$,
	and each has diameter less than $\eta$.
	Consequently, each is contained either in $h(K)$ or in
	$K\setminus h(K)$.

	Set
	\begin{equation*}
	\mathcal J=
	\bigl\{
	w\in\{1,\ldots,n\}^{\ell}:
	\psi_w(K)\subseteq K\setminus h(K)
	\bigr\}.
\end{equation*}
	Then $\mathcal J$ is non-empty and
	\begin{equation*}
	K\setminus h(K)=\bigcup_{w\in\mathcal J}\psi_w(K),
\end{equation*}
	with pairwise disjoint sets in the union.
	Thus $ 	\{h\}\cup\{\psi_w:w\in\mathcal J\} $
	is a finite generating IFS of contracting similitudes of $K$
	satisfying the SSC.
	This proves (iii).
	Since $\Psi$ was arbitrary, the same construction also proves
	the final assertion.
\end{proof}

\section{A question on graph-directed systems}\label{sec:questions}

Theorem~\ref{thm:main} shows that a self-similar set cannot admit both
an SSC generating IFS and an OSC generating IFS which fails the SSC,
without any relation being imposed on the contraction ratios of the
two systems. It is natural to ask whether an analogous statement holds
for graph-directed systems when both the underlying graph and the
attractor vector are fixed. We use the standard setting of Mauldin and
Williams~\cite{MauldinWilliams1988} and Das and Edgar~\cite{DasEdgar2005}.

Let $G=(V,E)$ be a finite strongly connected directed multigraph with
$V\neq\varnothing$, such that each vertex is the initial vertex of at
least one edge.  For
 $e\in E$, write $i(e)$ and $t(e)$ for the initial and terminal
 vertices of $e$, respectively. To each $e\in E$, associate a
 contracting similitude $\phi_e:\R^d\to\R^d$. There is a unique family
 $(K_v)_{v\in V}$ of non-empty compact subsets of $\R^d$ such that
 \begin{equation*}
 K_v=\bigcup_{i(e)=v}\phi_e(K_{t(e)}),\qquad v\in V.
 \end{equation*}
 We call $(K_v)_{v\in V}$ the attractor of the graph-directed system.
 Here and below, the union is taken over all edges with initial vertex
 $v$, with parallel edges regarded as distinct edges.

 Recall that the graph-directed system is said to satisfy the open set
 condition (OSC) if there exist non-empty open sets $U_v\subset\R^d$ such that  for each fixed $v$, the sets $\phi_e(U_{t(e)})$, with $i(e)=v$,
 are pairwise disjoint subsets of $U_v$.
  It is said to satisfy the strong separation
 condition (SSC) if, for each $v\in V$, the sets $\phi_e(K_{t(e)})$, with $i(e)=v$,
 are pairwise disjoint.

 	\begin{question}\label{qu:graph-directed}
 		Let $\Phi=\{\phi_e\}_{e\in E}$ and $\Psi=\{\psi_e\}_{e\in E}$ be two
 		graph-directed systems of contracting similitudes associated with the
 		same finite strongly connected directed multigraph $G$. Suppose that
 		they have the same attractor $(K_v)_{v\in V}$, where each $K_v$ is
 		non-singleton. If $\Phi$ satisfies the graph-directed OSC and $\Psi$
 		satisfies the graph-directed SSC, does $\Phi$ necessarily satisfy the
 		graph-directed SSC?
 	\end{question}

 When $G$ consists of a single vertex, Question~\ref{qu:graph-directed}
 has an affirmative answer by Theorem~\ref{thm:main}. We remark that
 the proof of Theorem~\ref{thm:main} does not extend directly to the
 graph-directed setting.

	For $v\in V$ and $r>0$, let $ 	N_v(r)=N_{K_v}(r), $  and write $\mathbf N=(N_v)_{v\in V}$. Let $a_e$ and $b_e$ denote the
contraction ratios of $\phi_e$  and
 $\psi_e$, respectively, and write
 $\mathbf a=(a_e)_{e\in E}$ and $\mathbf b=(b_e)_{e\in E}$. For a vector
 $\mathbf F=(F_v)_{v\in V}$ of functions on $(0,\infty)$, define

 \begin{equation*}
 (T_{\mathbf a}\mathbf F)_v(r)
 =
 \sum_{e\in E: i(e)=v}
 F_{t(e)}(r/a_e),
 \qquad
 (T_{\mathbf b}\mathbf F)_v(r)
 =
 \sum_{e\in E: i(e)=v}
 F_{t(e)}(r/b_e).
 \end{equation*}

 Since $\Psi$ satisfies the graph-directed SSC, for all sufficiently
 small $r>0$, $  \mathbf N(r)=(T_{\mathbf b}\mathbf N)(r). $

 Set $ \mathbf D=T_{\mathbf a}\mathbf N-\mathbf N. $    For all
 sufficiently small $r>0$, we have
 \begin{equation*}
 \mathbf D
 =
 T_{\mathbf a}T_{\mathbf b}\mathbf N-T_{\mathbf b}\mathbf N,
 \qquad
 T_{\mathbf b}\mathbf D
 =
 T_{\mathbf b}T_{\mathbf a}\mathbf N-T_{\mathbf b}\mathbf N.
 \end{equation*}

 In the one-vertex case, the dilation operators commute. The same
 commutation underlies the interchange of the two families of
 dilations in Lemma~\ref{lem:strict-counting} and in the integral
 calculation in Section~\ref{sec:proof}. For a general graph-directed
 system, however, $T_{\mathbf a}$ and $T_{\mathbf b}$ need not
 commute. Thus the identity $\mathbf D=T_{\mathbf b}\mathbf D$
 does not follow from the displayed formulas, and the scalar
 counting and integral argument does not directly resolve
 Question~\ref{qu:graph-directed}.

 \medskip
\noindent\textbf{Acknowledgements}. This work was supported by NKFIH,
Grant K144059, K152858. The author thanks K\'aroly Simon for his helpful
suggestions on the exposition of the first draft.

\section*{Disclosure of AI assistance}
The author used ChatGPT (OpenAI) to improve the language and
organization of the manuscript, assist in checking intermediate arguments,
and identify relevant references. The author independently checked the
mathematical arguments and cited references, reviewed and edited the
AI-assisted material, and takes full responsibility for the content
of the paper.


\begin{thebibliography}{99}

	\bibitem{Algom2020}
	A.~Algom,
	\emph{Affine embeddings of Cantor sets in the plane},
	J. Anal. Math. \textbf{140} (2020), no.~2, 695--757.

\bibitem{BandtGraf1992}
	C.~Bandt, S.~Graf,
	\emph{Self-similar sets. VII. A characterization of self-similar fractals with positive Hausdorff measure},
	Proc. Amer. Math. Soc. \textbf{114} (1992), no.~4, 995--1001.

\bibitem{BaranySimonSolomyak2023}
	B.~B\'ar\'any, K.~Simon, B.~Solomyak,
	{\em Self-similar and Self-affine Sets and Measures},
	Mathematical Surveys and Monographs, vol.~276,
	American Mathematical Society, Providence, RI, 2023.

	\bibitem{BishopPeres2017}
	C.~J.~Bishop, Y.~Peres,
	{\em Fractals in Probability and Analysis},
	Cambridge Studies in Advanced Mathematics, vol.~162,
	Cambridge University Press, Cambridge, 2017.

	\bibitem{DasEdgar2005}
	M.~Das, G.~A.~Edgar,
	\emph{Separation properties for graph-directed self-similar fractals},
	Topology Appl. \textbf{152} (2005), no.~1--2, 138--156.

\bibitem{DengLau2013}
	Q.-R.~Deng, K.-S.~Lau,
	\emph{On the equivalence of homogeneous iterated function systems},
	Nonlinearity \textbf{26} (2013), no.~10, 2767--2775.

\bibitem{DengLau2017}
	Q.-R.~Deng, K.-S.~Lau,
	\emph{Structure of the class of iterated function systems that generate the same self-similar set},
	J. Fractal Geom. \textbf{4} (2017), no.~1, 43--71.

\bibitem{ElekesKeletiMathe2010}
	M.~Elekes, T.~Keleti, A.~M\'ath\'e,
	\emph{Self-similar and self-affine sets: measure of the intersection of two copies},
	Ergodic Theory Dynam. Systems \textbf{30} (2010), no.~2, 399--440.

\bibitem{Falconer2003}
	K.~J.~Falconer,
	{\em Fractal Geometry: Mathematical Foundations and Applications},
	second edition, John Wiley \& Sons, Chichester, 2003.

	\bibitem{FengRuanXiong2023}
	D.-J.~Feng, H.-J.~Ruan, Y.~Xiong,
	\emph{On separation properties for iterated function systems of similitudes},
	Methods Appl. Anal. \textbf{30} (2023), no.~1, 17--26.

\bibitem{FengWang2009}
	D.-J.~Feng, Y.~Wang,
	\emph{On the structures of generating iterated function systems of Cantor sets},
	Adv. Math. \textbf{222} (2009), no.~6, 1964--1981.

\bibitem{Hutchinson1981}
	J.~E.~Hutchinson,
	\emph{Fractals and self-similarity},
	Indiana Univ. Math. J. \textbf{30} (1981), no.~5, 713--747.

\bibitem{MauldinWilliams1988}
	R.~D.~Mauldin, S.~C.~Williams,
	\emph{Hausdorff dimension in graph directed constructions},
	Trans. Amer. Math. Soc. \textbf{309} (1988), no.~2, 811--829.

\bibitem{RaoRuanYang2008}
	H.~Rao, H.-J.~Ruan, Y.-M.~Yang,
	\emph{Gap sequence, Lipschitz equivalence and box dimension of fractal sets},
	Nonlinearity \textbf{21} (2008), no.~6, 1339--1347.

\bibitem{Schief1994}
	A.~Schief,
	\emph{Separation properties for self-similar sets},
	Proc. Amer. Math. Soc. \textbf{122} (1994), no.~1, 111--115.

\bibitem{Xiao2024}
	J.-C.~Xiao,
	\emph{On a self-embedding problem for self-similar sets},
	Ergodic Theory Dynam. Systems \textbf{44} (2024), no.~10, 3002--3011.

\end{thebibliography}
\end{document}